\documentclass{amsart}
\usepackage[all]{xy}
\usepackage[latin1]{inputenc}
\usepackage{hyperref}
\usepackage{wasysym}

\newcommand{\C}{{\mathbb C} }
\newcommand{\R}{{\mathbb R} }

\newcommand{\cO}{{\mathcal O} }
\newcommand{\cP}{{\mathcal P} }

\newcommand{\cT}{{\mathcal T} }

\newcommand{\cX}{{\mathcal X} }

\newcommand{\wh}{\widehat}
\newcommand{\pt}{\partial}
\def\ol#1{{\overline{#1}}}
\newtheorem{theorem} {Theorem} [section]
\newtheorem{definition}[theorem] {Definition}
\newtheorem{lemma}[theorem]  {Lemma}

\newtheorem{proposition}[theorem] {Proposition}
\newtheorem{corollary}[theorem] {Corollary}

\def\ke{K{\"a}h\-ler-Ein\-stein }
\def\ks{Ko\-dai\-ra-Spen\-cer }
\def\ka{K{\"a}h\-ler}

\def\wp{Weil-Pe\-ters\-son }
\def\fn{Fen\-chel-Nielsen }
\def\tei{Teich\-mül\-ler }
\def\ii{\sqrt{-1}}

\def\C{\mathbb{C}}

\def\cinf{C^\infty}

\def\tr{{\mathrm{\, tr\, }}}

\def\ab{{\alpha\ol\beta}}

\def\gba{{g^{\ol\beta\alpha}}}
\def\gab{{g_{\alpha\ol\beta}}}
\def\na{\nabla_}

\begin{document}

\title[Weil-Petersson Symplectic Form]{Geometric Approach to the Weil-Petersson Symplectic Form}
\author[R.\ Axelsson]{Reynir Axelsson}
\address{Háskóli Íslands, Hjar{\tiny \DH}arhaga 6, IS-107 Reykjavík, Ísland}
\email{reynir@raunvis.hi.is}
\author[G. Schumacher]{Georg Schumacher}
\address{Fachbereich Mathematik und Informatik,
Philipps-Universit\"at Marburg, Lahnberge, Hans-Meerwein-Strasse, D-35032
Marburg,Germany}
\email{schumac@mathematik.uni-marburg.de}
\date{}

\begin{abstract}
In a family of compact, canonically polarized, complex manifolds the first
variation of the lengths of closed geodesics is computed. As an application, we show the coincidence of the \fn and \wp symplectic forms on the \tei spaces of
compact Riemann surfaces in a purely geometric way. The method can also be
applied to situations like moduli spaces of weighted punctured Riemann surfaces, where the methods of Kleinian groups are not available.
\end{abstract}

\maketitle

\section{Introduction}\label{sec:intro}
In 1958 André Weil introduced the methods of deformation theory of compact complex manifolds to the study of \tei spaces \cite{weil}. In particular, he used the $L^2$-inner product, induced by a metric of constant negative curvature, to the harmonic representatives of \ks classes. He conjectured that the resulting Hermitian inner product on the \tei space would be of \ka\ type and of negative curvature. The resulting \ka\ form is since known as the {\it \wp form}.

Beginning in 1982, Wolpert analyzed the \fn structure of the \tei space
and showed that the standard symplectic form $\omega^{FN} =- \sum d\tau^i\wedge d\ell^i$, defined in terms of \fn coordinates, actually coincides with the \wp form $\omega^{WP}$. This was achieved in a sequence of papers \cite{wolpert:fn-annals,wolpert:sym-annals,wolpert:amj}, which also contain important arithmetic results. Most of the arguments use in an essential way the theory of Kleinian groups. For example, one key idea in Wolpert's approach is  to solve Beltrami equations on the upper half plane, to compute
derivatives of cross-ratios of the solutions, and then use Poincaré series, in
order to descend to the Riemann surfaces.

Our aim is to extend these results to cases where the original methods do not apply, i.e.\ to cases where the hyperbolic structure is not necessarily induced by the universal covering, such as weighted punctured Riemann surfaces. Our methods also provide a purely differential geometric approach to a key result in Wolpert's approach, namely the Weil-Petersson duality of the twist vector fields $\partial/\partial\tau^j$ and the differentials $d\ell^j$ of the corresponding length functions. We use path integrals along closed geodesics of harmonic Beltrami differentials and holomorphic quadratic differentials. More generally, the first variation of the length of a closed geodesic in a holomorphic family of \ke manifolds of negative curvature can be expressed as the integral of the harmonic \ks tensor along a geodesic; this is our Theorem~\ref{th:varlength}.

We consider the \tei space of (possibly weighted) punctured Riemann surfaces of genus $p$ with $n$ punctures. The existence of pants decompositions and \fn coordinates was shown by Zhang in \cite{zhang} under the natural assumption that the cone angles do not exceed the value of $\pi$, i.e.\ the weights are in the interval $[1/2,1]$ (c.f.\ Section~\ref{sec:cone}). Let $s^\kappa; \kappa=1, \ldots, N$ be local holomorphic coordinates on the \tei space near a fixed point $s_0$ corresponding to a Riemann surface $X$;  then the tangent vectors $\partial/\partial s^\kappa|_{s_0}$ correspond to harmonic Beltrami differentials $\mu_\kappa$ on $X$. An immediate consequence of Theorem~\ref{th:varlength} is the following result:

\begin{theorem}\label{th:main1} The complex derivative of a length
coordinate with respect to a holomorphic coordinate equals half the integral of the corresponding harmonic Beltrami differential along the closed geodesic:
\begin{equation}\label{eq:main1}
\left.\frac{\partial \ell^i(s)}{\partial s^\kappa}\right|_{s_0}=
\frac{1}{2}\int_{\gamma_i} \mu_\kappa.
\end{equation}
\end{theorem}
The above formula is independent of the choice of a set of \fn coordinates and holds for the length function $\ell^\gamma(s)$ of a differentiable family $\gamma(s)$ of closed geodesics.

In the classical case, closely related results, where the first variation of the length of a closed geodesic is expressed in terms of integrals, can be found in \cite{wolpert:comment,wolpert:fn-annals}.

The shearing technique for twists along closed geodesics plays an important role in \tei theory. We handle the twist parameters $\tau^i$ in the context of deformation theory and cocycles of vertical automorphisms. This reflects only one aspect of the complex earthquake map associated to a given closed geodesic in the sense of McMullen \cite{mm}  (cf.\ also \cite{e-m}).
The tangent vector $\partial/\partial \tau^i|_{s_0}$ of the \tei space is uniquely determined by its action on the cotangent space, i.e.\ the space of holomorphic quadratic differentials. (Here we use the type decomposition of the real tangent space of the \tei space as embedded into its complexification.)

\begin{theorem}\label{th:main2}
The tangent vector $\partial/\partial \tau^i|_{s_0}$ applied to a quadratic
holomorphic differential $\varphi=\varphi(z)dz^2$ is given by the integration
of the quadratic differential along the corresponding closed geodesic.
\begin{equation}\label{eq:main2}
\left(\left.\frac{\partial}{\partial \tau^i}\right|_{s_0}\!\!, \varphi(z)
dz^2\right)= \frac{\ii}{2} \int_{\gamma_i} \varphi.
\end{equation}
\end{theorem}
Note that also the above formula is independent of the choice of \fn coordinates. It holds for the twist vector $\pt/\pt \tau^\gamma$ for any given closed geodesic $\gamma$.

Observe that the above inner product is actually a product of the $(1,0)$-part
of $(\partial/\partial \tau^i)$ and the quadratic holomorphic differential
$\varphi$. As a corollary we obtain a key result proved by Wolpert for the classical case in \cite{wolpert:fn-annals}.

\begin{theorem}\label{cor:FNWP}
The \fn symplectic form on the \tei space is equal to the \wp form:
\begin{equation}\label{eq:FNWP}
\omega^{FN}=\omega^{WP} \text{, i.e. } -\sum_{i=1}^N d\tau^i\wedge d\ell^i =
\ii\; G_{\kappa\ol\lambda}^{WP}\; ds^\kappa\wedge ds^\ol\lambda.
\end{equation}
\end{theorem}
In fact the formulas \eqref{eq:main1}, \eqref{eq:main2} and \eqref{eq:FNWP} are intimately related in the sense that any two of them imply the third one.

In the situation of Fuchsian groups, the above geodesic integral \eqref{eq:main2} seems to be related to the monodromy integral of the quotient of the given quadratic differential by a certain Abelian differential in \cite{wolpert:fn-annals} (cf.\ Section~\ref{sec:twist}).

Our approach can also be used to compute second variations -- our computations and applications will appear elsewhere.

{\it Acknowledgements.}
The second named author would like to thank Inkang Kim for important
discussions. The authors would like to thank the referee for pointing out the reference \cite{e-m}.

This paper was written at the University of Iceland. The second named author would like to express his thanks for the kind hospitality and support.

\section{Families of \ke Manifolds}
Let $\{\cX_s\}_{s\in S}$ be a holomorphic family of canonically polarized
compact complex manifolds parameterized by a (connected) complex space $S$. It
is given by a proper, smooth, holomorphic mapping $f:\cX \to S$ such that
$\cX_s=f^{-1}(s)$ for all $s\in S$. For simplicity we will assume that the base $S$ is smooth, although our results can also be given a meaning for possibly non-reduced singular base spaces.

We denote by $X=\cX_{s_0}$;  $s_0 \in S$ a distinguished fiber. Let
$n=\dim_\C X$, and denote by $z^\alpha$; $\alpha=1,\ldots, n$ local
holomorphic coordinates on $X$.

A Kähler form on $X$ will be denoted by
$$
\omega_X = \ii \gab dz^\alpha \wedge dz^\ol\beta
$$
On the fiber $X$ we are using the summation convention together with the $\nabla$-notation for covariant derivatives. A $|$--symbol will denote an ordinary derivative. Also, $\partial_\alpha$ and $\partial_\ol\beta$ will stand for $\partial/\partial z^\alpha$ and $\partial/\partial z^\ol\beta$ respectively. The raising and lowering of indices is defined as usual in terms of covariant derivatives.

For the Ricci tensor $R_{\alpha\ol\beta}$ of $X$ we will use the sign
convention
$$
R_\ab = - \log(g(z))_{|\ab},
$$
where $g(z)=\det(\gab(z))$.

Accordingly we set $g(z,s)= \det (g_\ab(z,s))$ for the family $f: \cX \to S$, where we equip the fibers $\cX_s$ with  \ke forms
$$
\omega_{\cX_s} = \ii g_\ab(z,s) dz^\alpha\wedge dz^\ol\beta
$$
of constant negative
curvature $-1$, i.e.\ $R_\ab(z,s) = - g_\ab(z,s)$.

Let
$$
\rho :T_{s_0}S \to H^1(X, \cT_X)
$$
be the \ks map for the corresponding deformation of $X$ over $S$ at the point
$s_0\in S$.

A natural inner product on the space $H^1(X, \cT_X)$ of infinitesimal
deformations of $X$, the {\it \wp}Hermitian inner product on $T_{s_0}S$, is
induced by the \ke metric $\omega_X$ on $X$. Namely, given tangent vectors $u,v \in T_{s_0}S$, we denote by $A_u= A_{u\ol\beta}^\alpha \partial_\alpha
dz^\ol\beta$ and $A_v$ the harmonic representatives of $\rho(u)$ and
$\rho(v)$ respectively. Then the inner product of $u$ and $v$ equals
$$
\langle u, v \rangle_{WP} = \int_X A_{u\ol\beta}^\alpha A_{\ol v
\gamma}^\ol\delta g_{\alpha\ol\delta}g^{\ol\beta\gamma} g \/ dV,
$$
where $A_\ol v$ denotes the adjoint (conjugate) tensor of $A_v$, and
$g\/ dV$ the volume element.

We note that the \wp inner product is positive definite at a given
point of the base, if the induced deformation is effective.

Let $s_j$; $j=1,\ldots, N$ be local coordinates on $S$, and set
$A_j= A_{\partial/\partial s_j}$. Then the \wp form on $S$ equals
$$
\omega^{WP} = \ii G_{i\ol\jmath}^{WP}(s) ds^i \wedge ds^\ol\jmath,
$$
where we use the notation
$$
G_{i\ol\jmath}^{WP}(s)= \langle\partial /\partial s^i, \partial
/\partial s^j\rangle_{WP} = \int_X A_{i\ol\beta}^\alpha A_{\ol
\jmath \gamma}^\ol\delta g_{\alpha\ol\delta}g^{\ol\beta\gamma} g \/
dV.
$$

A closed real $(1,1)$-form on the total space $\cX$ of the given
family is defined by
\begin{equation}\label{eq:omegaX}
\omega_\cX = - \ii \partial\ol\partial\log g(z,s).
\end{equation}
The \ke condition for the fibers implies that the restrictions of
$\omega_\cX$ to the fibers are the given \ke metrics:
$$
\omega_\cX|_{\cX_s}= \omega_{\cX_s}
$$
for all $s\in S$.

An important fact is that the family of \ke metric tensors contains
the harmonic representatives of \ks classes. The short exact
sequence
$$
0 \to \cT_{\cX/S} \to \cT_{\cX} \to f^*\cT_S \to 0
$$
induces the \ks map via the edge homomorphism for direct images. Again, we
use $\nabla$-notation for covariant derivatives as well as raising and
lowering of indices in fiber direction for families. A lift of a tangent vector
$\partial/\partial s^i$ of $S$ at $s_0$ as a differentiable vector
field of $\cX $ on $X$ is of the form
$$
\partial/ \partial s^i + b_i^\alpha \partial_\alpha.
$$
Its exterior $\ol\partial$-derivative $B_{i\ol\beta}^\alpha
\partial_\alpha dz^\ol\beta$, where $B_{i\ol\beta}^\alpha=\na\ol\beta b_{i}^\alpha
$, is interpreted as a $\ol\partial$-closed $(0,1)$-form on $X$ with
values in the tangent bundle of $X$. It represents the obstruction
against the existence of a holomorphic lift of the given tangent
vector, i.e.\ the infinitesimal triviality of the deformation.
Moreover,
\begin{equation}\label{eq:ks}
\rho(\partial/\partial s^i)= [B_{i\ol\beta}^\alpha \partial_\alpha
dz^\ol\beta] \in H^1(X,\cT_X).
\end{equation}

Our argument is based upon the notion of canonical lifts in the sense of Siu
(cf.\ \cite{siu:canlift}), which turned out to be $\omega_\cX$-horizontal: The
form $\omega_\cX$ is positive definite in fiber directions but need not be
positive on $\cX$. Its components in the directions of $z^\alpha$ and $s^j$ are denoted by $g_{\alpha\ol\jmath}$ etc. However, horizontal lifts of tangent
vectors are well-defined and of the form
\begin{equation}\label{eq:lift}
\partial/\partial s^i + a_i^\alpha \partial_\alpha,
\end{equation}
where
\begin{equation}\label{eq:liftcomp}
a_i^\alpha= -\gba g_{i\ol\beta}.
\end{equation}
We set $A_{i\ol\beta}^\alpha = \na\ol\beta a_{i}^\alpha$ and obtain (cf.
\cite{sch:curv,sch:teich}) the following fact:

\begin{proposition}\label{pr:harm}
The horizontal lifts with respect to $\omega_\cX$ of tangent vectors
induce harmonic representatives of \ks classes: The harmonic
representative of $\rho(\partial/\partial s^i)$ equals
$A_{i\ol\beta}^\alpha \partial_\alpha dz^\ol\beta$. These satisfy
the following properties
\begin{itemize}
\item[(i)]
$$
\na\ol\delta A^\alpha_{i\ol\beta}=\na\ol\beta A^\alpha_{i\ol\delta},
$$

\item[(ii)]
$$
\na\gamma A_{i\ol\beta}^\alpha g^{\ol\beta\gamma}=0,
$$
\item[(iii)]
$$
A_{i\ol\beta\ol\delta}=A_{i\ol\delta\ol\beta}.
$$
\end{itemize}
\end{proposition}
The conditions (i) and (ii) above correspond to harmonicity, whereas  condition (iii) reflects the relationship with the metric tensor.

\section{Families of closed geodesics}
Let $(f: \cX \to S, \omega_\cX)$ be a family of \ke manifolds with
constant negative Ricci curvature equal to $-1$, where $\omega_\cX$ is given by \eqref{eq:omegaX}.

We denote by $\gamma_s$ a differentiable family of closed geodesics
contained in the fibers $\cX_s$. In order to describe the variation of the length of closed geodesics in a family, we use the notion of integrating a tensor along a geodesic. Exemplarily we define:
\begin{definition}
Let $C=C_{\ol\beta\ol\delta}$ be a tensor on the Kähler manifold $X$, and
$\gamma$ be a geodesic of length $\ell$, parameterized by $u(t)= (u^1(t),\ldots,u^n(t))$, such that $\|\dot u(t)\|=1$. Then
$$
\int_\gamma C = \int_\gamma C_{\ol\beta\ol\delta}dz^\ol\beta
dz^\ol\delta := \int_0^\ell C_{\ol\beta\ol\delta}(u(t)) \dot
u^\ol\beta \dot u^\ol\delta dt.
$$
\end{definition}
For contravariant tensors of order one this notation coincides with
the integration of a differential form along the curve $\gamma$.

The main result of this section is the following theorem:

\begin{theorem}\label{th:varlength}
Let $\cX \to S$ be a family of canonically polarized, compact,
complex manifolds equipped with \ke metrics. Let $\gamma_s$ be a
differentiable family of closed geodesics on the fibers $\cX_s$, having the lengths $\ell(s)$. Let $\partial/\partial s^i|_{s=s_0} \in T_{s_0}S$
be a tangent vector. Then the first variation of the length of the
closed geodesic with respect to the holomorphic parameter of the
family is half the integral of the symmetric, harmonic \ks tensor along
the geodesic:
$$
\left.\frac{\partial}{\partial s^i}\right|_{s=s_0}\!\!\!\ell(s)=
\frac{1}{2} \int_{\gamma(s_0)} A_{i\ol\beta\ol\delta}dz^\ol\beta
dz^\ol\delta. 
$$
\end{theorem}

In order to prove the theorem, we may assume that $S=\{s\}$ is a
disk in $\C$, and $s_0=0$. The closed geodesic curves
$\gamma_s$ are solutions of the equation for geodesics
\begin{equation}\label{eq:geodesic}
\ddot u^\alpha(t,s) + \Gamma^\alpha_{\gamma\sigma}(u(t,s))\dot
u^\gamma(t,s)\dot u^\sigma(t,s) =0.
\end{equation}
The solution is unique up to the constant value of the speed
$$
\|\dot u(t,s)\|^2= {\gab(u(t,s),s)\dot u^\alpha(t,s)\dot
u^\ol\beta(t,s)}.
$$
For $s=0$ we chose $\|\dot u\|=1$, for the remaining values of $s$ the value of $\|\dot u\|$ will be determined by the fact that the parameter $t$ assumes
values in the interval $[0, \ell_0]$, where $\ell_0$ is the length of
$\gamma_0$. Now
$$
\ell(s)= \int_0^{\ell_0} \|\dot u(t,s)\| dt
$$
so that
\begin{equation}\label{eq:dls}
\left.\frac{d\ell(s)}{ds}\right|_{s=0}= \frac{1}{2} \int_0^{\ell_0}
\frac{d}{ds}\|\dot u(t,s)\| ^2 dt.
\end{equation}
We need to study various tensors on $\cX$ along the geodesics. In
particular, we consider the vector field
$$
U_s(t,s) = \left.\left(\frac{\partial}{\partial s} +u_s^\alpha(t,s)
\frac{\partial}{\partial z^\alpha} + u_s^\ol\beta(t,s)\frac{\partial}{\partial
z^\ol\beta}\right)\right|_{u(t,s)},
$$
where
$$
u_s^\alpha(t,s) =\frac{\partial u^\alpha(t,s)}{\partial s} \text{
and } u_s^\ol\beta(t,s) =\frac{\partial u^\ol\beta(t,s)}{\partial
s}\,.
$$
The function
$$
\chi(t,s) = \langle U_s, \dot u \rangle
$$
is defined as the inner product of vector fields along the geodesics with
values on the total space $\cX$. Concerning $\omega_\cX$ we indicate the
direction of the one-dimensional space $S$ by $s$. We have
$$
\chi(t,s) = g_{s\ol\beta} \dot u^\ol\beta + \gab u_s^\alpha\dot
u^\ol\beta + \gab \dot u^\alpha u_s^\ol\beta.
$$
Using \eqref{eq:geodesic} we simplify the expression for the
derivative $\dot \chi$ and get:
\begin{equation}\label{eq:chidot}
\begin{split}
\dot \chi(t,s) = g_{s\ol\beta|\alpha} \dot u^\alpha \dot u^\ol\beta
+ g_{s\ol\beta|\ol\delta} \dot u^\ol\beta \dot u^\ol\delta +
g_{s\ol\beta} \ddot u^\ol\beta  + g_{\alpha\ol\beta|\gamma} \dot
u^\gamma u_s^\alpha \dot u^\ol\beta +
\\
\gab \dot u_s^\alpha \dot u^\ol\beta +
g_{\alpha\ol\beta|\ol\delta}\dot u^\alpha u_s^\ol\beta \dot
u^\ol\delta + \gab \dot u^\alpha \dot u_s^\ol\beta.
\end{split}
\end{equation}
\begin{lemma}\label{le:diffchidot}
$$
\frac{\partial}{\partial s}\|\dot u\|^2 - \dot \chi = A_{s\ol\beta\ol\delta}
\dot u^\ol\beta\dot u^\ol\delta.
$$
\end{lemma}
\begin{proof}
We use \eqref{eq:geodesic} and \eqref{eq:chidot} together with
\eqref{eq:lift} and \eqref{eq:liftcomp} and apply
Proposition~\ref{pr:harm}. The statement now follows from a
calculation.
\end{proof}
The {\it proof} of Theorem~\ref{th:varlength} concludes with
\eqref{eq:dls} and Lemma~\ref{le:diffchidot}.

\section{Cocycles of vertical automorphisms}\label{sec:vaut}
Let $X$ be a compact complex manifold, and $(S,s_0)$ a germ of a complex space. We follow the approach of Forster and Knorr \cite{fo-kn} and assign to any
deformation of $X$ over a $(S,s_0)$ an element of the first cohomology
$H^1(X,\mathfrak{G}_S)$ of $X$ with values in the sheaf $\mathfrak{G}_S$ of
vertical automorphisms. Its derivative with respect to the base space yields
the \ks map. For any deformation of $X$ given by a cartesian diagram
$$
\xymatrix{X \ar[r] \ar[d] & \cX \ar[d]^f\\ 0 \ar[r] & (S,s_0)}
$$
there exist open coverings $\{U_i\}$ of $X$  and $\{Z_i\}$ of $\cX$ together
with isomorphisms $\psi_i: U_i \times S \to Z_i$, compatible with the cartesian diagram, and equal to the identity for $s=0$. The actual cocycles $\gamma_{ij}$ are defined on $(U_i\cap U_j )\times S$ by $\psi_j^{-1}\circ \psi_i$. We look at the derivatives of the $\gamma_{ij}$ with respect to tangent vectors of the
base and values in the sheaf of holomorphic vector fields. These define the \ks classes and will be compared with suitable $(0,1)$-forms representing the same
classes.

We take a differentiable trivialization  $\zeta :\cX \to X \times S$ of the
family, which equals the identity for $s=s_0$. Then, according to \eqref{eq:ks}, the class $\rho(\partial/\partial s^\kappa)|_{s=s_0}$ is represented by
\begin{equation}\label{eq:ksdol}
\left.\frac{\partial^2\zeta^\alpha(z,s)}{\partial s^\kappa\partial
z^\ol\beta}\right|_{s=s_0}\!\frac{\partial}{\partial
z^\alpha}dz^\ol\beta.
\end{equation}
We now consider the situation where the covering that is involved
in the deformation only consists of only two elements, $U_1$ and $U_2$,
say. A differentiable trivialization of the holomorphic family can
be chosen to be equal to the identity on the complement of any given
compact neighborhood of $Z_1\cap Z_2\subset \cX$.

\section{Application to Riemann surfaces}
For families of Riemann surfaces (with notation from Section~\ref{sec:intro})
the above Proposition~\ref{pr:harm} states that harmonic Beltrami differentials are of the form
\begin{equation}\label{eq:muharm}
\mu_\kappa=\mu_\kappa(z)\frac{\partial}{\partial z} \ol{dz}= -
\frac{\partial}{\ol{\partial z}}\left(\frac{1}{g} \frac{\partial^2 \log
g}{\partial s^\kappa \ol{\partial z}}\right)\frac{\partial}{\partial z}
\ol{dz}.
\end{equation}
We write these also in the form
\begin{equation}\label{eq:harmqua}
\mu_\kappa(z)\frac{\partial}{\partial z} \ol{dz} =
\frac{\varphi_\ol\kappa(z)}{g(z)} \frac{\partial}{\partial z} \ol{dz}
\end{equation}
for certain quadratic holomorphic differentials
$\varphi_\kappa=\varphi_\kappa(z)dz^2$, with $\varphi_\ol\kappa
:=\ol{\varphi_\kappa}$.

Now Theorem~\ref{th:main1} is an immediate consequence of Theorem~\ref{th:varlength}.

\section{Twist coordinates vs.\ holomorphic coordinates}\label{sec:twist}
In this section we will prove Theorem~\ref{th:main2}. We will first treat the
case of Riemann surfaces $X$ equipped with the hyperbolic metrics
$g=g(z)|dz|^2$. Given the context, we normalize a metric of constant curvature
$-1$ so that the Ricci curvature equals $-1$, i.e.\ $\partial^2\log
g(z)/\partial z \ol{\partial z}= g(z)$.

Our point is to give a direct proof, which only uses the hyperbolic geometry of Riemann surfaces.

We consider a closed geodesic $\gamma$ on $X$ parameterized by
$u(t)$ with $\|\dot u(t)\|=1$. Let the length of $\gamma$ be $\ell_0
= 2\pi r_0$. The twist along $\gamma$ gives rise to a real, one
parameter family of Riemann surfaces, in particular it gives rise to
a real (analytic) curve in the appropriate \tei space.

As local coordinate set we use an annulus $W=\{z\in \C ; r_1 < |z| < r_2\}$
embedded into $X$, such that $\gamma$ corresponds to a circle $\{ |z|= r_0 \}$
for $r_1<r_0<r_2$, with parameterization $u(t)= r_0 e^{\ii t/r_0}$ for $0 \leq
t \leq \ell_0 $.

For $\tau \in \R$ a family of holomorphic automorphisms of $W$ is given by the
twist $z \mapsto  e^{\ii \tau/r_0} z$. In the sense of Section~\ref{sec:vaut}
we interpret these as a family of vertical automorphisms with respect to an
open covering of $X$ consisting of two elements. Now a differentiable
trivialization (which is equal to the identity on the complement of $W$) is
given by
\begin{equation}\label{eq:triv}
\zeta(z,\tau)= e^{\ii \frac{\tau}{r_0} h(r)}z,
\end{equation}
where $r=|z|$ and $h\in \cinf((r_1,r_2),\R)$ denotes a differentiable
function which is identically zero near $r_1$, and identically equal to one
near $r_2$. The value $\tau= \ell_0$ represents a Dehn twist.

We take a deformation theoretic standpoint, and we use in this context the shearing technique \cite[Section~1]{wolpert:fn-annals}: A Beltrami differential will again stand for an infinitesimal deformation. If we replace in \eqref{eq:ksdol} a tangent vector $\partial/\partial s^\kappa$ by a real tangent vector
$\partial/\partial\tau$, we obtain the image of its $(1,0)$-component under the \ks map. The resulting Beltrami differential $\mu=\mu(z)(\partial/\partial z)
\ol{dz}$ is identically zero on the complement of $W$. For  $z \in W$ it
follows immediately that
\begin{equation}\label{eq:beltau}
\mu(z)=\frac{\ii}{2} \frac{z^2}{r r_0}h'(r).
\end{equation}

Now we can prove Theorem~\ref{th:main2}. The idea is to determine a Beltrami differential as a linear form on the cotangent space of the \tei space.

The duality  $H^1(X,\cT_X) \times H^0(X, \cO_X(2K_X))\to \C$ of tangent and
cotangent spaces of $\cT$ is given by
$$
\left(\mu(z)\frac{\partial}{\partial z} \ol{dz},
\varphi(z)dz^2\right) \mapsto \int_X \mu(z) \varphi(z) dA,
$$
where $dA =\ii dz \wedge \ol{dz}$ is the Euclidean area element. Observe that
$\mu$ need not be harmonic in this expression, in fact we get the same value
for any function $h(r)$ with the above properties. The claim of the theorem
follows from \eqref{eq:beltau} in principle by Stokes' theorem -- we can also
let $h$ tend to the function $h_0$, which is defined to be equal to zero for
$r<r_0$ and equal to one for $r\geq r_0$.  Then $h'(r)$ converges to a delta function, and integration over $r$ implies that that for any quadratic
holomorphic differential $\varphi$
$$
(\mu,\varphi)= \frac{\ii}{2}\;r_0 \int_0^{2\pi} e^{2\ii \vartheta} \varphi(r_0
e^{\ii \vartheta})\/ \/d\vartheta.
$$
Now the necessary reparameterization $\vartheta=t/r_0$ yields the claim of the
theorem. \qed

We mention the monodromy integrals from \cite{wolpert:fn-annals}. If $A$ is a hyperbolic element in $PSL(2,\R)$, lifted to $\wh A\in SL(2,\R)$, the monodromy integral is
$$
\cP(\psi,\wh A)=\frac{1}{2} \int_s^{As} \psi \Omega_{\wh A}
$$
where $\psi$ is a quadratic holomorphic differential, and $\Omega_{\wh A}^{-1}$ a suitable Abelian differential. Then by \cite[Lemma~4.2]{wolpert:fn-annals} the following formula holds for the twist tangent vector $\mu_0$ applied to a holomorphic quadratic differential
$$
\int_\Delta \mu_0 \psi =- \ii \sigma_{\wh A} {\eta(\wh A)}^{-1}\cP(\psi,\wh A),
$$
where $\Delta$ is a fundamental domain, $\sigma_{\wh A}=\pm 1$, and $\eta(\wh A)= ((\tr \wh A)^2-4)^{1/2}$.

\section{Computation of the \fn symplectic form}
The results of this section are valid for the \tei space of (possibly weighted) punctured Riemann surfaces.

Combining our geodesic integral formulas \eqref{eq:main1} and \eqref{eq:main2} we immediately obtain a result, which was proved in the classical case by Wolpert (\cite[Theorem 2.10]{wolpert:fn-annals}).

\begin{corollary}\label{cor:main1}
The vector fields $\partial/\partial \tau^j$ are Hamiltonian with respect to the \wp form:
\begin{eqnarray}
  \frac{\partial\ell^j}{\partial s^\ol\kappa}&=&-\ii \left\langle \frac{\partial}{\partial
  \tau^j}, \mu_\kappa\right\rangle, \label{eq:ltaudual1} \\
  \frac{\partial\ell^j}{\partial s^\kappa}&=& \ii \left\langle \mu_\kappa , \frac{\partial}{\partial
  \tau^j}\right\rangle. \label{eq:ltaudual2}
\end{eqnarray}
\end{corollary}
Here, the \wp inner product, which is defined initially on the complex tangent
space, is extended in a natural way to the complexified real tangent space
(which also contains the complex tangent vectors of type $(0,1)$).


Next, we prove the coincidence \eqref{eq:FNWP} of the \fn and \wp symplectic
forms using \eqref{eq:main1} and \eqref{eq:main2}. We will give a short
computational argument:

\medskip

\noindent {\it Proof of Theorem~\ref{cor:FNWP}.} We write $\omega^{WP}$ in
terms of \fn coordinates and begin with the contribution of the twist parameter involving $d\tau^j\wedge d\tau^k$. Since we are dealing with real forms, we
need to show:

\smallskip
\noindent {\it Claim.} {\it The real part
$$
\text{\rm Re}\left(\ii\; \left\langle \frac{\pt}{\pt s^\alpha},\frac{\pt}{\pt s^\beta}
\right\rangle  \frac{\pt s^\alpha}{\pt \tau^j} \frac{\pt s^\ol\beta}{\pt
\tau^k}\right)
$$
vanishes, in particular, it is symmetric in $j$ and $k$.}

\medskip

\noindent {\it Proof of the claim}:
\begin{eqnarray*}
\ii\; \left\langle \frac{\pt}{\pt s^\alpha},\frac{\pt}{\pt s^\beta}
\right\rangle  \frac{\pt s^\alpha}{\pt \tau^j} \frac{\pt s^\ol\beta}{\pt
\tau^k} &=& \ii \; \left\langle \frac{\pt}{\pt s^\alpha},\frac{\pt s^\beta}{\pt
\tau^k} \frac{\pt}{\pt s^\beta} \right\rangle  \frac{\pt s^\alpha}{\pt
\tau^j}\\
&=& \ii \; \left\langle \frac{\pt}{\pt s^\alpha},\frac{\pt}{\pt \tau^k}
\right\rangle
\frac{\pt s^\alpha}{\pt \tau^j} \text{\footnotesize\quad (type consideration) }\\
& =\atop\eqref{eq:ltaudual2}& \frac{\pt \ell^k}{\pt s^\alpha}\cdot \frac{\pt
s^\alpha}{\pt \tau^j}.
\end{eqnarray*}
The real part of this quantity equals
$$
\frac{1}{2} \frac{\pt\ell^k}{\pt\tau^j}=0.\eqno{\square}
$$

\smallskip
We conclude the proof of Theorem~\ref{cor:FNWP}. The vanishing of the
coefficients of $d\ell^k\wedge d\ell^m$ is an immediate consequence of the
\lq\lq dual\rq\rq\ equation \eqref{eq:detaues}, which may not be available at this point. However, we can revert here to Wolpert's purely (differential) geometric
argument \cite[Corollary 1.2]{wolpert:amj}, which involves orientation
reversing (anti-holomorphic) symmetries.

Finally, the coefficient of $d\tau^j\wedge d\ell^m$ in $\omega^{WP}$ equals
\begin{eqnarray*}
&&\hspace{-15mm} \ii \; \left(\left\langle \frac{\pt s^\alpha}{\pt \tau^j}
\frac{\pt}{\pt s^\alpha}, \frac{\pt s^\beta}{\pt \ell^m}\frac{\pt}{\pt s^\beta}
\right\rangle - \left\langle \frac{\pt s^\alpha}{\pt \ell^m} \frac{\pt}{\pt
s^\alpha}, \frac{\pt s^\beta}{\pt \tau^j}\frac{\pt}{\pt s^\beta} \right\rangle
\right)\\
&=& \ii \; \left(\left\langle \frac{\pt }{\pt \tau^j} , \frac{\pt s^\beta}{\pt
\ell^m}\frac{\pt}{\pt s^\beta} \right\rangle - \left\langle \frac{\pt
s^\alpha}{\pt \ell^m} \frac{\pt}{\pt s^\alpha}, \frac{\pt }{\pt \tau^j}
\right\rangle \right)\\
&=\atop{\eqref{eq:ltaudual1}\eqref{eq:ltaudual2}}& - \left(\frac{\pt
\ell^j}{\pt s^\ol\beta}\frac{\pt s^\ol\beta}{\pt \ell^m}+ \frac{\pt \ell^j}{\pt
s^\alpha}\frac{\pt s^\alpha}{\pt \ell^m}\right)\\
&=& -\delta_{\ell m} . \hspace{6cm} \qed
\end{eqnarray*}

Now we obtain the following formula of Wolpert from the equality $\omega^{FN}=\omega^{WP}$:

\begin{corollary}
The length coordinates $\ell^j$ give rise to Hamiltonian vector fields
$\partial/\partial \ell^j$:
\begin{equation}\label{eq:detaues}
\frac{\partial \tau^j}{\partial s^\alpha}=-\ii \left\langle
\frac{\partial}{\partial s^\alpha},\frac{\partial}{\partial \ell^j}
\right\rangle.
\end{equation}
\end{corollary}

The above equation \eqref{eq:ltaudual2} also implies the following corollary, whose classical version is Theorem~2.11 from \cite{wolpert:fn-annals}: We first observe that \eqref{eq:ltaudual2} holds in a slightly more general situation. Namely, as usual, we pick local holomorphic
coordinates $(s_1,\ldots, s_N)$ on the \tei space near a point $s_0$ and denote by $\mu_\alpha$ the harmonic Beltrami differentials corresponding to the
$\pt/\pt s^\alpha|_{s_0}$. Then for any closed geodesic $\gamma$ on $X$ we
have the length function $\ell(\gamma)$, and also $\pt/\pt \tau(\gamma)$ is defined at the given point $s_0$ without a choice of a system of \fn coordinates. Now \eqref{eq:ltaudual2} reads
\begin{equation}\label{eq:ltaudualvar}
\frac{\pt \ell(\gamma)}{\pt s^\alpha}= \ii \; \left\langle \frac{\pt }{\pt
s^\alpha}, \frac{\pt}{\pt \tau(\gamma)}\right\rangle.
\end{equation}
for all $\alpha=1,\ldots, N$.

\begin{corollary} Let $\gamma$ and $\delta$ be closed geodesics on $X$. Then
\begin{equation}\label{eq:wol2.11}
\frac{\partial \ell(\gamma)}{\partial \tau(\delta)} = - \frac{\partial \ell(\delta)}{\partial\tau(\gamma)}.
\end{equation}
\end{corollary}
\begin{proof} We show that the left-hand side is skew symmetric in $\gamma$ and $\delta$.
\begin{eqnarray*}
\frac{\partial \ell(\gamma)}{\partial \tau(\delta)} &=&2 \text{\rm Re}\left(
\frac{\partial\ell(\gamma)}{\pt s^\alpha}\cdot \frac{\pt s^\alpha}{\pt \tau(\delta)} \right)\\
&=\atop\eqref{eq:ltaudualvar}& 2 \text{\rm Re} \left(\ii\;\frac{\pt s^\alpha}{\pt \tau(\delta)}\cdot \left\langle\frac{\pt}{\pt s^\alpha},\frac{\partial}{\pt \tau(\gamma)}  \right\rangle \right)\\
&=& 2 \text{\rm Re} \left(\ii\;\frac{\pt s^\alpha}{\pt \tau(\delta)}\cdot\left\langle\frac{\pt}{\pt s^\alpha},\frac{\partial s^\beta}{\pt\tau(\gamma)}\frac{\pt}{\pt s^\beta} \right\rangle \right) \text{\footnotesize\quad (type consideration) }\\
&=& 2 \text{\rm Re} \left( \ii\; G^{WP}_{\alpha\ol\beta}\frac{\pt s^\alpha}{\pt\tau(\delta)} \frac{\pt s^\ol\beta}{\pt \tau(\gamma)} \right)\\
&=& \ii \; G^{WP}_{\alpha\ol\beta} \left(\frac{\pt s^\alpha}{\pt \tau(\delta)}
\frac{\pt s^\ol\beta}{\pt \tau(\gamma)} - \frac{\pt s^\alpha}{\pt \tau(\gamma)}
\frac{\pt s^\ol\beta}{\pt \tau(\delta)}\right).
\end{eqnarray*}
\end{proof}

\section{Weighted punctured Riemann surfaces and conical metrics}\label{sec:cone}
Let $X$ be a compact Riemann surface of genus $p$ and ${\bf a} = \sum_{i=1}^{n} a_i p_i$ be an $\R$-divisor, with $0<a_i\leq 1$ for all $i$. Then $(X,\bf a)$ is called a weighted punctured Riemann surface. We always assume that the degree of the $\R $-divisor $ K_{(X, {\bf a})} := K_X + { \bf a}$, is positive. It was shown by Hulin, McOwen and Troyanov \cite{h-t:curv,Mc-O2,Troy} that such surfaces possess unique hyperbolic conical metrics. These are characterized as follows:
\begin{itemize}
\item[(i)] For $a_j=1$  the metric $g$ satisfies a Poincaré growth condition at $p_j$,  i.e.\ $g = \frac{ \rho(z) }{ |z|^2 \log^2(|z|^2)}|d z|^2 $ in some local holomorphic coordinate system, where $p_j$ corresponds to $z=0$, and $\rho$ is a continuous positive function.
\item[(ii)] For $ 0 < a_j <1$ the metric is of the form $g = \frac{ \rho(z) }{|z|^{2a_j}}|dz|^2$, where $\rho$ is continuous.
\end{itemize}
In a holomorphic family, in particular regarding \tei and moduli spaces,
punctures are labeled. The complex structure of the resulting \tei space
$\cT_{\bf a}$ of weighted punctured surfaces (with weights being fixed) can be
identified with $\cT_{p,n}$. However, weights give rise to a hierarchy of
compactifications of the corresponding moduli spaces, which were extensively
investigated by Hassett \cite{Ha:modwe}.

On the other hand, hyperbolic conical metrics induce different \ka\ structures
on \tei and moduli spaces. Generalized \wp metrics were introduced and studied
in \cite{sch-tra, sch:conic}.

From the viewpoint of hyperbolic geometry, the \tei space of weighted punctured Riemann surfaces were studied by Zhang in \cite{zhang}. Under the seemingly
necessary assumption that the weights are between $1/2$ and $1$, Zhang showed
the existence of \fn coordinates so that the Fenchel-Nielsen symplectic form
$\omega^{FN}$ becomes meaningful. The above range of weights contains the
interesting weights of the form $1 - 1/k$ that arise from finite quotients. Any pants decomposition of a Riemann surface with conical singularities with the
above restriction avoids the punctures. So Theorems~\ref{th:main1} and
\ref{th:main2} are still valid.

Concerning the notion of harmonic Beltrami differentials, Hilbert space theory
is not available in general. Instead we use equation \eqref{eq:harmqua} as a
definition for the space $H^1(X,\bf a)$ of such differentials on a weighted
punctured Riemann surface $(X, \bf a)$. In \cite{sch-tra} the existence of a
duality
\begin{equation}
H^1(X,{\bf a}) \times H^0(X, \cO_X(2K_X + D)) \to \C,
\end{equation}
was shown, where $D= \sum_i p_i$,  is the underlying divisor to $\bf a$ with
weights one. Also \eqref{eq:muharm} holds. The $L^2$ inner product on
$H^1(X,{\bf a})$ induced by the conical hyperbolic metric gives rise to a
generalized \wp metric, which is of \ka\ type \cite{sch-tra,sch:conic}.
The statement and proof of the Theorem~\ref{cor:FNWP} can be carried over litterally to this case. So all
necessary conditions for the coincidence of $\omega^{FN}$ and $\omega^{WP}$ are satisfied for weighted punctured Riemann surfaces.

\end{document}